\documentclass[12pt]{article}

\usepackage{amsmath}
\usepackage{amsfonts}
\usepackage{amssymb}
\usepackage{graphicx, color}
\usepackage[colorlinks]{hyperref}
\usepackage{epsfig}
\usepackage{epstopdf} 
\usepackage{mathtools}
\usepackage{comment}
\usepackage[shortlabels]{enumitem}

\newtheorem{theorem}{Theorem}[section]
\newtheorem{lemma}[theorem]{Lemma}

\newtheorem{remark}[theorem]{Remark}
\newtheorem{example}[theorem]{Example}

\newtheorem{definition}[theorem]{Definition}

\newcommand{\recip}{\displaystyle\frac{1}}
\newcommand{\Frac}{\displaystyle\frac}
\newcommand{\nl}{\\[5pt]}

\newcommand{\inR}{\in\R}
\newcommand{\Int}{\displaystyle\int}
\newcommand{\R}{\mathbb{R}}
\newcommand{\CL}{\mathcal{L}}

\newcommand{\tu}{\tilde{u}}
\newcommand{\tv}{\tilde{v}}

\newcommand{\tH}{\tilde{H}}

\newcommand{\tJ}{\tilde{J}}
\newcommand{\tK}{\tilde{K}}

\newcommand{\tsigma}{\tilde{\sigma}}

\usepackage{mathrsfs}
\usepackage{array}

\usepackage{xcolor}
\begin{document}
\title{Stability of p-area minimizing surfaces in the Heisenberg group}

\author{ 
    \qquad Amir Moradifam\footnote{Department of Mathematics, University of California, Riverside, California, USA. E-mail: amirm@ucr.edu. Amir Moradifam is supported by NSF grants DMS-1715850 and DMS-1953620.}
    \qquad Gerardo Orozco-Fernandez \footnote{Department of Mathematics, University of California, Riverside, California, USA. E-mail: goroz003@ucr.edu.   }
}

\date{\today}
\smallbreak \maketitle

\begin{abstract}
We study the stability of minimizers of weighted $p$-area functionals associated with prescribed $p$-mean curvature surfaces in the Heisenberg group. While existence and uniqueness results are well established, quantitative stability with respect to perturbations of the mean curvature $H$ remains largely unexplored in the nonzero-$H$ regime.

Using a Rockafellar--Fenchel duality framework, we identify a unique underlying vector field associated with each minimizer and prove its stability under perturbations of $H$. This yields quantitative control of the direction field of the horizontal gradient. Building on this structure, we establish $L^1$ stability of admissible minimizers under natural geometric assumptions on level sets. In dimensions two and three, we also derive $W^{1,1}$ stability estimates under additional regularity and structural hypotheses, with explicit rates in terms of $\|H-\tilde H\|_{L^\infty}$.

Our results provide the first quantitative stability theory for $p$-area minimizing graphs with prescribed nonzero $p$-mean curvature, even in the unweighted case. Numerical simulations are included to illustrate the robustness of the theoretical results.
\end{abstract}

\section{Introduction}

The study of minimal and constant mean curvature surfaces is a central theme in geometric analysis and the calculus of variations. 
In the Euclidean or Riemannian setting, surfaces with zero mean curvature locally minimize area, while those with constant nonzero mean curvature describe equilibrium interfaces under pressure or volume constraints. 
A natural question is how these notions extend to \emph{sub-Riemannian} or \emph{pseudohermitian} geometries, where only a horizontal distribution of admissible directions is available and the metric is determined by the Carnot--Carath\'eodory distance.

The Heisenberg group arises naturally in several areas of physics where geometry is constrained or noncommutative. It provides the geometric framework underlying the position–momentum relations in quantum mechanics, reflecting the fundamental noncommutativity of observables. It also models systems with restricted motion, such as nonholonomic mechanics and anisotropic media, where only certain directions of movement are allowed. In addition, sub-Riemannian structures like the Heisenberg group appear in wave propagation and geometric optics, as well as in energy minimization problems involving interfaces in constrained environments. These connections make it a fundamental model for studying physical phenomena where classical Euclidean geometry is insufficient. The Heisenberg group $\mathbb{H}^n$ provides the simplest and most fundamental example of such a geometry. 
It is a step-two nilpotent Lie group endowed with a natural horizontal subbundle and a noncommutative structure that profoundly influences the geometry of curves and surfaces. 
In this setting, the appropriate analogue of classical mean curvature is the \emph{$p$-mean curvature}, defined as the divergence of a normalized horizontal gradient.
The nonparametric form of the prescribed $p$-mean curvature equation for graphs $u:\Omega\subset\mathbb{R}^{2n}\to\mathbb{R}$ is
\begin{equation}\label{eq:prescribedH}
\nabla \cdot \left( \frac{\nabla u - X^*}{|\nabla u - X^*|} \right) = H 
\quad \text{in } \Omega,
\end{equation}
where $X^*=(y_1,-x_1,\dots,y_n,-x_n)$.  
This equation arises as the Euler--Lagrange equation for the $p$-area functional
\begin{equation}\label{eq:functional}
E(u) = \int_{\Omega} \bigl(\,|\nabla u - X^*| + H\,u \bigr)\,dx,
\qquad a(x)>0,
\end{equation}
whose minimizers represent \emph{surfaces of prescribed $p$-mean curvature}~$H$ in~$\mathbb{H}^n$. 
Here $X^*$ encodes the intrinsic rotation of horizontal directions in $\mathbb{H}^n$.  
The term $\nabla u - X^*$ corresponds to the horizontal gradient of the defining function $\phi(x,y,t)=t-u(x,y)$, and $|\nabla u - X^*|$ therefore measures the horizontal slope of the graph.  
The divergence in~\eqref{eq:prescribedH} expresses that the surface has the prescribed $p$-mean curvature~$H$. The equation \eqref{eq:prescribedH} with $X\equiv  0$ and $H\equiv 0$ have been extensively studied by many authors, including the first author, see \cite{HMN, JMN, Mo, Mo1, MNT, MNTa_SIAM, NTT07, NTT08, NTT10, NTT11,  sternberg_ziemer92, sternberg_ziemer, sternbergZiemer93, ST}.

Early progress on this problem was achieved by Pauls~\cite{P}, who studied the $p$-minimal surface case ($H=0$) in~$\mathbb{H}^1$.  
He formulated the associated subelliptic equation, proved existence for certain Plateau-type boundary data, and showed that uniqueness can fail, since distinct $p$-minimal graphs may share the same boundary.  
This phenomenon reflects the non-strict convexity of the $p$-area functional and illustrates the delicate analytical nature of the problem.

Cheng, Hwang, Malchiodi, and Yang~\cite{CHMY} developed the geometric and analytic foundations of $p$-minimal surfaces in pseudohermitian manifolds.  
They introduced intrinsic definitions of $p$-mean curvature and $p$-area, elucidated the role of characteristic curves, and analyzed the mixed elliptic--hyperbolic character of the governing equation.  
Their results show that the $p$-mean curvature encodes fundamental geometric information of pseudohermitian structures and connects naturally to calibration theory and Cauchy–Riemann (CR) geometry (see also \cite{CHP}). 

Building on this framework, Cheng, Hwang, and Yang~\cite{CHY} established existence and uniqueness for minimizers of the functional
\begin{equation}\label{eq:functional-nonzeroH}
E(u) = \int_{\Omega} \bigl(|\nabla u - X^*| + H\,u \bigr)\,dx,
\end{equation}
under suitable boundary slope and convexity conditions.  
Their analysis treats the full prescribed $p$-mean curvature equation with nonzero~$H$, addressing the degeneracy of the PDE along the characteristic set 
\(\Sigma(u)=\{x\in\Omega:\nabla u=X^*(x)\}\) 
and the coupling introduced by the linear term~$H\,u$.  
The results in~\cite{CHY} treat the full prescribed $p$-mean curvature equation with 
nonzero~$H$, showing that this case requires analytical techniques beyond those used for 
the minimal ($H=0$) setting and constitutes an essential part of the general theory.  
In the broader sub-Riemannian context, surfaces with constant $p$-mean curvature are closely 
connected to the \emph{isoperimetric problem} in~$\mathbb{H}^n$, where minimizers of the 
horizontal perimeter under fixed volume are expected to have boundaries of constant 
$p$-mean curvature.

The nonzero-$H$ regime thus occupies a central position in pseudohermitian geometry.  
Geometrically, it corresponds to sub-Riemannian analogues of constant mean curvature surfaces, while analytically it introduces substantial challenges: the coupling term $H\,u$ links geometry and potential, breaks homogeneity, and amplifies the effects of degeneracy where $|\nabla u - X^*|$ vanishes.  
Nonuniqueness and limited regularity can persist, and the dependence of minimizers on the data $H$ becomes a subtle and important issue.

In \cite{MR}, the authors studued existence and structure of minimizers of the energy functional $\mathbb{E}$ from a different point of view, using the Rockafellar-Fenchel duality, and  proved various existence results for the weighted form of the functional
\begin{equation}\label{mainFunctionalWeighted}
\mathcal{F}(u)=\int_{\Omega} \left(a |\nabla u -X^*|+Hu \right)dx,
\end{equation}
where $a \in L^{\infty}(\Omega)$ is a positive function. Minimizers of this functional will satisfy the Euler-Lagrange equation 
\begin{equation}\label{EL1}
    \nabla \cdot \left(a \frac{\nabla u-X^*}{|\nabla u-X^*|} \right)=H,
\end{equation}
which could be viewed as the $p$-mean curvature of the function $(X, u(X))$, with respect to the metric $g=a^{\frac{2}{n-1}}dx$, which is conformal to the Euclidean metric (see also \cite{MR2} for a generalization of the results).

Although the existence and uniqueness theory is well developed, the question of \emph{stability}, that is, how minimizers depend quantitatively on perturbations of $H(x)$, remains largely open when $H \neq 0$. Understanding this dependence is essential both for theoretical completeness and for the reliability of numerical schemes based on variational or operator splitting methods. The aim of the present paper is to investigate the stability of minimizers of the functional \eqref{mainFunctionalWeighted} with respect to the mean curvature $H$. Building on the geometric and analytic foundations developed in~\cite{CHMY,CHY,LM,LM2,MR,MR2,P}, we establish new stability and quantitative continuity results for minimizers of the $p$-area functional with prescribed nonzero $H$.

These results provide a more refined understanding of the structure of pseudohermitian $p$-mean curvature graphs and extend the existing theory beyond existence and uniqueness to include stability. They are new even in the special case $a \equiv 1$. We also carry out numerical simulations to illustrate the robustness of the results.

\section{Duality and Stability of the Underlying Vector Field}
Let $\Omega$ be a connected, open, and bounded subset of $\R^n$ with Lipschitz boundary. Consider the minimization problem 
\begin{equation}\label{mainmin0}
    \inf\limits_{w\in H^1_0(\Omega)}\Int_\Omega(a|Dw+F|+Hw),
\end{equation}
and note that the more general minimization problem on $H^1_{f}(\Omega)=\{w\in H^1(\Omega): w=f \ \ \text{on} \ \ \partial \Omega\}$ can be reduced to \eqref{mainmin} with a simple change of variable $w=w_0+f$, where $w_0 \in H^1_0(\Omega)$ and $f \in H^1(\Omega) \cap H^1_f(\Omega)$.

By Rockafellar-Fenchel duality \cite{ET}, the dual problem associated to \eqref{mainmin0} can be written as 
\begin{equation}
   (D) \hspace{0.5cm} \sup \{\langle F,b \rangle: b\in \mathcal{D}_0 \ \ \hbox{and}\ \ |b| \leq a \ \ \hbox{a.e. in}\ \ \Omega \},
\end{equation}
where 
\begin{equation}
    \mathcal{D}_0:=\left \{b\in (L^2(\Omega))^n: \int_{\Omega} \nabla u \cdot b+Hu =0, \ \ \hbox{for all}\ \ u\in H^1_0(\Omega)   \right \}.
\end{equation}

Indeed $b \in (L^{\infty}(\Omega))^n \cap \mathcal{D}_0$ if and only if 
\[\nabla \cdot b= H \ \ \hbox{a.e. in}\ \ \Omega.\]
See section 3.1 in \cite{MR} for more details. The following result shows that there is a unique underlying vector field $J$ that determines the direction of $Du+F$ in $\Omega$, for any minimizer $u$.

\begin{theorem} \label{Structure1} \cite{MR} \label{Structure2}
Let $\Omega$ be a bounded domain in $\R^n$, $F,H \in L^{2}(\Omega)$, $a \in L^{2}(\Omega)$ be a positive function, and assume \eqref{mainmin0} is bounded below. Then the duality gap is zero and the dual problem $(D)$ has a solution, i.e. there exists a vector field $J \in \mathcal{D}_0$ with $|J|\leq a$, $|Du+F|-a.e.$ in $\Omega$, such that 
\begin{equation}\label{dualityGap2}
\inf _{u\in H^1_0(\Omega) } \int_{\Omega} \left( a \left| D u + F \right| +Hu \right) dx= \langle F, J\rangle
\end{equation}
Moreover 
\begin{equation}\label{directionParallel2}
a \frac{Du+F}{|Du+F|}= J, \ \ \ \ |Du+F|-a.e. \ \ \hbox{in}\ \ \Omega,
\end{equation}
for any minimizer $u$ of \eqref{mainmin0}.  
\end{theorem}

\begin{remark}
Note that the \eqref{mainmin0} may not have a minimizer, but the dual problem $(D)$ always has a solution $J \in (L^2(\Omega))^n$. In addition, the minimization problem \eqref{mainmin0} is not strictly convex, and it may have multiple minimizers (see \cite{JMN}), and the equality \eqref{directionParallel2} holds for every minimizer of \eqref{mainmin0}.
\end{remark}

This theorem and the underlying vector field $J$ will play an essential role in our stability arguments. Now let $u$ be a minimizer of 
\begin{equation}\label{mainmin}
    \inf\limits_{w\in BV_0(\Omega)}\Int_\Omega(a|Dw+F|+Hw).
\end{equation}
Note that $u$ satisfies the Euler-Lagrange equation
\begin{equation}\label{EL}
    \nabla\cdot\left(a\frac{Du+F}{|Du+F|}\right)=H,\hspace{0.1in}u|_{\partial\Omega}=0.
\end{equation}
Then there exits a vector field $J$, independent of the choice of the minimizer $u$, such that 
\[a \frac{Du+F}{|Du+F|}= J, \ \ \ \ |Du+F|-a.e. \ \ \hbox{in}\ \ \Omega,\]
with $\nabla \cdot J=H$. Moreover  
\[J(x):=\sigma(x)(Du(x)+F(x)) \quad \text{and} \quad |J|=a, \ \ \ \ |Du+F|-a.e. \ \ \hbox{in}\ \ \Omega, \]
for some function $\sigma \geq 0$. In particular, 
$$
    a=|J|=\sigma|Du+F|=J\cdot\frac{(Du+F)}{|Du+F|}, \ \ \ \ |Du+F|-a.e. \ \ \hbox{in}\ \ \Omega.
$$
Similarly, if $\tu$ be a minimizer of 

\begin{equation}\label{mainmin1}
    \inf\limits_{w\in BV_0(\Omega)}\Int_\Omega(a|Dw+F|+\tH w),
\end{equation}
where $\tH$ is a perturbation of $H$, then one can similarly define $\tJ$ and $\tsigma$.

    Next we aim to show that the structure of the level sets of minimizers of the least gradient problem \eqref{mainmin} is stable. The arguments are partly inspired by our earlier results in \cite{LM}. Throughout the paper, we will assume that $a\in L^\infty(\Omega)$ and
    \begin{equation}\label{mbounds}
        0<m\leq a(x)\leq M,\,\,x\in\Omega,
    \end{equation}
    for some positive constants $m,M.$

It is well known that (see \cite{G, MB}) there exists a constant $C_\Omega>0$ depending only on $\Omega$ such that, $\text{for all }u\in BV_0(\Omega),$
 \begin{equation}\label{C_Omega}
      \|u\|_{L_1(\Omega)}\leq C_\Omega\int_\Omega |Du|.
 \end{equation}

We shall need the following Lemmas. 

\begin{lemma} \label{rmk H2}
    Suppose
    $$
        \|F\|_{L^1(\Omega)}\leq k_1
    $$
    for some $k_1>0,$ and 
    \begin{equation}\label{Hbound}
        \|H\|_{L^\infty(\Omega)}<\Frac{m}{C_\Omega}
    \end{equation}
    where $C_\Omega$ is the constant in \eqref{C_Omega}. If $u$ is a minimizer of \eqref{mainmin1}, then 
    \begin{equation}\label{uppbound}
    \|u\|_{L^1(\Omega)}\leq C,
    \end{equation}
    for some constant $C,$ depending on $\Omega,\,m,$ and $k_1,$ independent of $u$ and $\tu.$ Consequently, the optimization problem \eqref{mainmin1} is bounded away from $-\infty$.

\end{lemma}
{\bf Proof.} We have 
\allowdisplaybreaks
\begin{align*}
    \|u\|_{L^1(\Omega)}
    &\leq C_{\Omega}\Int_{\Omega}|Du|\,dx\nl
    &\leq C_{\Omega}\Int_{\Omega}|Du+F|\,dx+C_{\Omega}\Int_{\Omega}|F|\,dx\nl
    &\leq\frac{C_{\Omega}}{m}\Int_{\Omega}a|Du+F|\,dx+k_1C_{\Omega}\nl
    &\leq\frac{C_{\Omega}}{m}\Int_{\Omega}a|Du+F|+Hu\,dx-\frac{C_{\Omega}}{m}\Int_{\Omega}Hu\,dx+k_1C_{\Omega}\nl
    &\leq\frac{C_{\Omega}C_1}{m}+\frac{C_{\Omega}}{m}\|H\|_{L^\infty(\Omega)}\|u\|_{L^1(\Omega)}+k_1C_{\Omega}.
\end{align*}
where $C_1=\left|\Int_{\Omega}a|Du+F|+Hu\,dx\right|.$ Thus 
\[ \|u\|_{L^1(\Omega)} \leq\frac{C_{\Omega}C_1+mk_1C_{\Omega}}{m -C_{\Omega}\|H\|_{L^\infty(\Omega)}},\]and hence  $\|u\|_{L^1(\Omega)}$ is bounded above by a constant independent of $u$. 
\hfill $\Box$

\begin{lemma}\label{lemma H1} Let $u$ and $\tu$ be minimizers of \eqref{mainmin} and \eqref{mainmin1}. Then 
$$
    \left|
        \int_\Omega a|Du+F|+Hu\,dx
        -\int_\Omega a|D\tu+F|+\tH\tu\,dx
    \right|
    \leq C\|H-\tH\|_{L^\infty(\Omega)},
$$
for some constant $C=C(\Omega,m,k_1)$ independent of $u$ and $\tu.$
\end{lemma}
{\bf Proof.}
\begin{align*}
    \Int_\Omega Hu\,dx-\Int_\Omega \tH u\,dx
    &=\Int_\Omega a|Du+F|+Hu\,dx-\Int_\Omega a|Du+F|+\tH u\,dx\nl
    &\leq\Int_\Omega a|Du+F|+Hu\,dx-\Int_\Omega a|D\tu+F|+\tH\tu\,dx\nl
    &\leq\Int_\Omega a|D\tu+F|+H\tu\,dx-\Int_\Omega a|D\tu+F|+\tH\tu\,dx\nl
    &=\Int_\Omega H\tu\,dx-\Int_\Omega \tH\tu\,dx.
\end{align*}
Thus,
\begin{align*}
    \Int_\Omega (H-\tH)u\,dx
    &\leq\Int_\Omega a|Du+F|+Hu\,dx-\Int_\Omega a|D\tu+F|+\tH\tu\,dx\nl
    &\leq\Int_\Omega (H-\tH)\tu\,dx.
\end{align*}
Hence 
\begin{align*}
    -\|H-\tH\|_{L^\infty(\Omega)}\|u\|_{L^1(\Omega)}
    &\leq\Int_\Omega a|Du+F|+Hu\,dx-\Int_\Omega a|D\tu+F|+\tH\tu\,dx\nl
    &\leq\|H-\tH\|_{L^\infty(\Omega)}\|\tu\|_{L^1(\Omega)},
\end{align*}
and therefore  
\begin{align*}
    &\left|
        \Int_\Omega a\big(|Du+F|-|D\tu+F|\big)+Hu-\tH\tu\,dx
    \right|\nl
    &\leq
    \max
    \left\{
        \|u\|_{L^1(\Omega)},\|\tu\|_{L^1(\Omega)}
    \right\}
    \|H-\tH\|_{L^\infty(\Omega)}.
\end{align*}
Note that similar to Lemma \ref{rmk H2}, $\|\tu\|_{L^1(\Omega)}$ is bounded above by a constant, and thus, the proof is complete. \hfill $\Box$ \\

Let $\nu_\Omega$ denote the outer unit normal vector to $\partial\Omega.$ Then for every $T\in\left(L^\infty(\Omega)\right)^n$ with $\nabla\cdot T\in L^n(\Omega),$ there exists a unique function $[T,\nu_\Omega]\in L^\infty(\partial\Omega)$ such that
\begin{equation}\label{wibp}
    \int_{\partial\Omega}[T,\nu_\Omega]u\,d\mathcal{H}^{n-1}=\int_\Omega u\nabla\cdot T\,dx+\int_\Omega T\cdot Du\,dx,\hspace{0.1in}u\in C^1(\bar\Omega).
\end{equation}
Moreover, for $u\in BV(\Omega)$ and $T\in\left(L^\infty(\Omega)\right)^n$ with $\nabla\cdot T\in L^n(\Omega),$ the linear functional $u\mapsto (T\cdot Du)$ gives rise to a Radon measure on $\Omega$ and \eqref{wibp} holds for all $u\in BV(\Omega)$ (see \cite{GA1,GA2}).

\begin{lemma}\label{lemma H2} Suppose that $0\leq\sigma(x)\leq\sigma_1$ for some constant $\sigma_1>0.$ Then
$$
    \int_\Omega |J||\tJ|-J\cdot\tJ\,dx\leq C\|H-\tH\|_{L^\infty(\Omega)},
$$
where $C=C(\Omega, m,k_1,\sigma_1)$ is a constant independent of $u$ and $\tu.$
\end{lemma}
{\bf Proof.}
Note that by the weak integration by parts formula,
\begin{align*}
    \int_\Omega u\nabla\cdot\tJ+\tilde J\cdot Du\,dx
    =\int_{\partial\Omega}[\tJ,\nu_\Omega]u\,dx
    =0
    =\int_{\partial\Omega}[\tJ,\nu_\Omega]\tu\,dx
    =\int_\Omega \tu\nabla\cdot\tJ+\tJ\cdot D\tu\,dx
\end{align*}
and thus,
$$
    \int_\Omega \tH u+\tJ\cdot Du\,dx=\int_\Omega \tH\tu+\tJ\cdot D\tu\,dx.
$$
Then
\allowdisplaybreaks
    \begin{align*}
        \int_\Omega |J||\tJ|-J\cdot\tJ\,dx
        &=\int_\Omega \sigma|\tJ||Du+F|-\sigma\tJ\cdot(Du+F)\,dx\nl
        &\leq\sigma_1\int_\Omega |\tJ||Du+F|-\tJ\cdot F-\tJ\cdot Du\,dx\nl
        &=\sigma_1\int_\Omega |\tJ||Du+F|-\tJ\cdot F-(-\tH u+\tH\tu+\tJ\cdot D\tu)\,dx\nl
        &=\sigma_1\int_\Omega |\tJ||Du+F|-\tJ\cdot(D\tu+F)+\tH u-\tH\tu\,dx\nl
        &=\sigma_1\int_\Omega |\tJ||Du+F|-|\tJ||D\tu+F|+\tH u-\tH\tu\,dx\nl
        &=\sigma_1\int_\Omega a|Du+F|+Hu-a|D\tu+F|-\tH\tu+\left(\tH-H\right)u\,dx\nl
        &\leq \sigma_1(C+\|u\|_{L^1(\Omega)})\|H-\tH\|_{L^\infty(\Omega)}
    \end{align*}
    where $C$ is as in Lemma \ref{lemma H1}. By Lemma \ref{rmk H2}, the proof is complete.
\hfill $\Box$

Now we are ready to prove the main result of this section. 
\begin{theorem}\label{thm H1}
    Suppose that $0\leq\sigma(x)\leq\sigma_1$ for some constant $\sigma_1>0.$ Then
$$
    \|J-\tJ\|_{L^1(\Omega)}\leq C\|H-\tH\|_{L^\infty(\Omega)}^{\frac{1}{2}},
$$
where $C=C(\Omega, m,k_1,\sigma_1)$ is a constant independent of $u$ and $\tu.$
\end{theorem}
{\bf Proof.} We have 
    \begin{align*}
        \Int_\Omega\sqrt{|J-\tJ|^2}\,dx
        &=\Int_\Omega\sqrt{|J|^2+|\tJ|^2-2J\cdot\tJ}\,dx\nl
        &=\Int_\Omega\sqrt{2a^2-2J\cdot\tJ}\,dx\nl
        &=\sqrt{2}\Int_\Omega\sqrt{|J||\tJ|-J\cdot\tJ}\,dx\nl
        &\leq\sqrt{2|\Omega|}\left(\Int_{\Omega}|J||\tJ|-J\cdot\tJ\,dx\right)^{1/2}.
    \end{align*}
Hence the result follows from Lemma \ref{lemma H2}.
\hfill $\Box$

\begin{remark} In view of Theorem \ref{Structure1}, $\frac{D u+F}{|Du+F|}$ and $\frac{D \tilde{u}+F}{|D\tilde{u}+F|}$ are parallel to $J$ and $\tilde{J}$, respectively.  So Theorem \eqref{thm H1} implies that  if $\tilde{H}$ is close to $H$, then the direction of $Du+F$ is close to the direction of $D \tu+F$ in $\Omega$. This result will play a crucial role in the proof of our results in the following sections.
\end{remark}

\section{$L^1$ stability of the minimizers}

In this section, we establish the stability of minimizers of the least gradient problem \eqref{mainmin} in $L^1(\Omega).$ In general, \eqref{mainmin} does not have unique minimizers so in order to prove any stability results, further assumptions on $H,\tH$ and on the corresponding minimizers are necessary.

\begin{definition}
    Fix positive constants $\sigma_0,\sigma_1.$ We say that $u\in C^1\left(\bar\Omega\right)$ is \textit{admissible} if it solves \eqref{mainmin} for some $\sigma\in C(\Omega)$ with
    $$
        0<\sigma_0<\sigma<\sigma_1,
    $$
    and $m\leq |J|=|\sigma(\nabla u+F)|\leq M,$ where $m$ and $M$ are positive constants as in \eqref{mbounds}. We similarly define admissibility for $\tu.$
\end{definition}

For the remainder of the paper, we will assume that $F$ is a conservative vector field $\big($i.e. $F=\nabla f$ for some function $f\in C^1(\Omega)\big).$ Furthermore, assume $f\in C(\bar\Omega)$ and on $\bar\Omega,$ define $v:=u+f$ and $\tv:=\tu+f.$ Then $v-\tv=u-\tu.$

We will prove our results in dimension $n=2$ and then extend them to dimension $n=3.$ Let $v\in C^1(\bar\Omega)$ with $|\nabla v|>0$ almost everywhere in $\Omega.$ By the regularity result of De Giorgi, (Theorem 4.11 in \cite{EG}) it follows that almost all level sets of $v$ are $C^1$ hypersurfaces. For $n=2,$ we will furthermore assume that the length of level sets of $v$ in $\Omega$ is uniformly bounded, i.e.
  \begin{equation}\label{crvbnd}
      \sup\limits_{t\inR}\int_{\{v=t\}\cap\Omega}\,dl=K<\infty.
  \end{equation}

  \begin{lemma}
      Almost every level set of $v$ reaches $\partial\Omega.$
  \end{lemma}
{\bf Proof.}
      Almost every level set is a $C^1$ hypersurface. Suppose that there exists $t\inR$ such that $\Gamma_0:=\{x\in\bar\Omega:v(x)=t\}$ satisfies $\Gamma_0\cap\partial\Omega=\varnothing.$ By the Alexander duality theorem (see, e.g., Theorem 27.10 in
\cite{GH}), $\R^n$ is partitioned into a bounded open connected region, $\Omega_t,$ and $\R^n\backslash\bar\Omega_t,$ which share the boundary $\Gamma_0.$ In $\bar\Omega_t,$ define $w(x)=v(x)-t.$ Then
      \begin{align*}
          \Int_{\Omega_t}a|Dw|+Hw
          \,dx
          &\geq m\|Dw\|_{L^1(\Omega_t)}-\|H\|_{L^\infty(\Omega)}\|w\|_{L^1(\Omega_t)}\nl
          &\geq \left(m-\|H\|_{L^\infty(\Omega)}C_{\Omega_t}\right)\|Dw\|_{L^1(\Omega_t)}
      \end{align*}
      where $C_{\Omega_t}$ is the constant in \eqref{C_Omega}. Extend the definition of $w$ to be identically 0 on $\bar\Omega\backslash\bar\Omega_t.$ Since $\|w\|_{L^1(\Omega_t)}=\|w\|_{L^1(\Omega)}\leq C_\Omega\|Dw\|_{L^1(\Omega)},$ we may assume that $C_{\Omega_t}\leq C_\Omega.$ Therefore, by \eqref{Hbound},
      $$
          \Int_{\Omega_t}a|Dw|+Hw
          \,dx
          \geq \left(m-\|H\|_{L^\infty(\Omega)}C_{\Omega}\right)\|Dw\|_{L^1(\Omega_t)}\geq0.
      $$
      Equivalently, $$
        \Int_{\Omega_t}a|Dv|+Hv\,dx
        \geq t\Int_{\Omega_t}H\,dx,
      $$
      and the inequality is strict if $v \not \equiv t$. Hence, $v\equiv t$ is the only minimizer of $\Int_{\Omega_t}a|Dv|+Hv\,dx$ in $\Omega_t.$ 

      Now consider the disjoint collection of all such $\Omega_t$ and, for each $t,$ choose $x_t\in\Omega_t\cap\mathbb{Q}^n.$ Then each $x_t$ is distinct. Thus, the number of level sets that do not reach the boundary is countable.
      \hfill$\Box$

  \begin{theorem}\label{thm H2}
      Let $n=2,$ and suppose $u$ and $\tu$ are admissible with $u|_{\partial\Omega}=0=\tu|_{\partial\Omega}.$ If $u$ satisfies \eqref{crvbnd} then
    $$
        \|u-\tu\|_{L^1(\Omega)}\leq C\|H-\tH\|_{L^\infty(\Omega)}^{\frac{1}{2}},
    $$
    for some constant $C(\Omega, m,M,k_1,K,\sigma_0,\sigma_1)$ independent of $\tu$ and $\tsigma.$
  \end{theorem}
{\bf Proof.}
    Since $u$ is admissible,
    $$
        0<\frac{m}{\sigma_1}\leq|\nabla u+F|=|\nabla v|
    $$
    for all $x\in\Omega.$ From the coarea formula, it follows that
    \begin{equation}\label{i1}
        \frac{m}{\sigma_1}\int_\Omega|u-\tu|\,dx
        =\frac{m}{\sigma_1}\int_\Omega|v-\tv|\,dx\leq\int_\Omega|\nabla v||v-\tv|\,dx=\int_\R\int_{\{v=t\}\cap\Omega}|v-\tv|\,dl\,dt.
    \end{equation}
    Let $\Gamma_t$ be a connected component of $\{x\in\Omega:v(x)=t\}$ and let $\gamma:[0,L]\to\Gamma_t$ be a path parameterized by the arc length of $\Gamma_t$ with $\gamma(0)\in\partial\Omega.$ Now define $h:[0,L]\to\R$ by
    $$
        h(s)=v(\gamma(s))-\tv(\gamma(s)).
    $$
    Then $h(0)=0.$ Since $v$ is constant on $\Gamma_t,$ then $0=\frac{d}{ds}v(\gamma(s))=\nabla v(\gamma(s))\cdot\gamma'(s)$ on $\Gamma_t.$ Therefore,
    \begin{align*}
        h'(s)
        &=\nabla v(\gamma(s))\cdot\gamma'(s)-\nabla \tv(\gamma(s))\cdot\gamma'(s)\nl
        &=\frac{\sigma(\gamma(s))}{\tsigma(\gamma(s))}\nabla v(\gamma(s))\cdot\gamma'(s)-\nabla \tv(\gamma(s))\cdot\gamma'(s)\nl
        &=\frac{J(\gamma(s))-\tJ(\gamma(s))}{\tsigma(\gamma(s))}\cdot\gamma'(s).
    \end{align*}
    Note that there exists $x_t^*$ on $\Gamma_t$ such that
    $$
        |v(x_t^*)-\tv(x_t^*)|=\max\limits_{x\in\Gamma_t}|v(x)-\tv(x)|.
    $$
    Furthermore, there exists $s_0\in[0,L]$ such that $x_t^*=\gamma(s_0).$ Then
    \begin{align*}
        |v(x_t^*)-\tv(x_t^*)|
        &=|h(s_0)|\nl
        &=\left|\int_0^{s_0}h'(r)\,dr\right|\nl
        &\leq\recip{\sigma_0}\int_0^L|J(\gamma(r))-\tJ(\gamma(r))|\,dr\nl
        &=\recip{\sigma_0}\int_{\Gamma_t}|J-\tJ|\,dl.
    \end{align*}
    Thus,
    $$
        \int_{\Gamma_t}|v-\tv|\,dl
        \leq |v(x_t^*)-\tv(x_t^*)|\int_{\Gamma_t}\,dl
        \leq\frac{K}{\sigma_0}\int_{\Gamma_t}|J-\tJ|\,dl
    $$
    and therefore,
    $$
        \int_{\{v=t\}\cap\Omega}|v-\tv|\,dl
        \leq\frac{K}{\sigma_0}\int_{\{v=t\}\cap\Omega}|J-\tJ|\,dl.
    $$
    By the coarea formula and Theorem \ref{thm H1},
    \begin{align*}
        \int_\R\int_{\{v=t\}\cap\Omega}|v-\tv|\,dl\,dt
        &\leq\frac{K}{\sigma_0}\int_\R\int_{\{v=t\}\cap\Omega}|J-\tJ|\,dl\,dt\nl
        &=\frac{K}{\sigma_0}\int_\Omega|\nabla v||J-\tJ|\,dx\nl
        &\leq\frac{KM}{\sigma_0^2}\int_\Omega|J-\tJ|\,dx\nl
        &\leq\frac{KMC}{\sigma_0^2}\|H-\tH\|_{L^\infty(\Omega)}^{1/2}
    \end{align*}
    for some constant $C=C(\Omega,m,k_1,\sigma_1).$ By \eqref{i1}, the proof is complete. \hfill $\Box$

\bigskip

In order to generalize Theorem \ref{thm H2} to dimension $n=3,$ we
need the following additional assumption on level sets of $v.$

\begin{definition}\label{def H2}
    Let $n=3.$ Let $u\in C^1(\bar\Omega)$ be admissible and $v$ be defined as in Remark \ref{rmk H2}. We say that the level sets of $v$ can be foliated to one-dimensional curves if for almost every $t\in \text{range}(v)$, every connected component $\Gamma_t$ of $\{v=t\},$ there exists a function $g_t(x)\in C^1(\Gamma_t)$ such that $0<c_g\leq |\nabla g_t|\leq C_g,$ for some constants $c_g, C_g$ independent of $t.$ Moreover, every component of $\{v=t\}\cap\{g_t=r\}\cap\Omega$ is a $C^1$ curve reaching $\partial\Omega$ for almost every $t\in\text{range}(v)$ and for all $r\inR.$ Assume that the lengths of the connected components of $\{v=t\}\cap\{g_t=r\}\cap\Omega$ are uniformly bounded by some constant $K.$
\end{definition}

\begin{theorem}\label{thm H3}
    Let $n=3$ and suppose that $u$ and $\tu$ are admissible with $u|_{\partial\Omega}=\tu|_{\partial\Omega}=0.$ Suppose that the level sets of $v$ can be foliated to one-dimensional curves as in Definition \ref{def H2}. Then
$$
    \|u-\tu\|_{L^1(\Omega)}\leq C\|H-\tH\|_{L^\infty(\Omega)}^{1/2}
$$
for some constant $C(\Omega, m,M,k_1,K,\sigma_0,\sigma_1,c_g,C_g)$ independent of $\tu$ and $\tsigma.$
\end{theorem}
{\bf Proof.}
    The proof is similar to the proof of Theorem \ref{thm H2}. Since $u$ is admissible, then
    \begin{equation}\label{I1}
        \frac{m}{\sigma_1}\int_\Omega|u-\tu|\,dx
        =\frac{m}{\sigma_1}\int_\Omega|v-\tv|\,dx
        \leq\int_\Omega|\nabla v||v-\tv|\,dx
        =\int_\R\int_{\{v=t\}\cap\Omega}|v-\tv|\,dS\,dt
    \end{equation}
    by the coarea formula. Consider $g_t$ from Definition \ref{def H2}. By the coarea formula, 
    \begin{align}\label{I2}
    \begin{split}
        \int_\R\int_{\{v=t\}\cap\Omega}|v-\tv|\,dS\,dt
        &=\int_\R\int_\R\int_{\{v=t\}\cap\{g_t=r\}\cap\Omega}\recip{|\nabla g_t|}|v-\tv|\,dl\,dr\,dt\nl
        &\leq\recip{c_g}\int_\R\int_\R\int_{\{v=t\}\cap\{g_t=r\}\cap\Omega}|v-\tv|\,dl\,dr\,dt.
    \end{split}
    \end{align}
    Parameterize the connected components $\Gamma_t$ of $\{v=t\}\cap\{g_t=r\}\cap\Omega$ by the arc length, $\gamma:[0,L]\to\Gamma_t$ where $L$ is the arc length of $\Gamma_t$ and $\gamma(0)\in\partial\Omega.$ Let $h(s)=v(\gamma(s))-\tv(\gamma(s)).$ Let $x_t^*$ be a point which maximizes $|v-\tv|$ on $\Gamma_t$ and let $s_0\in[0,L]$ be a point such that $\gamma(s_0)=x_t^*.$ Similar to the proof of Theorem \ref{thm H2},

    $$
        |v(x_t^*)-\tv(x_t^*)|\leq\recip{\sigma_0}\int_0^L|J(\gamma(\tau))-\tJ(\gamma(\tau))|\,d\tau
    $$
    and therefore,
    $$
        \int_{\Gamma_t}|v-\tv|\,dl\leq\frac{K}{\sigma_0}\int_{\Gamma_t}|J-\tJ|\,dl.
    $$
    Hence,
    \begin{equation}\label{I3}
        \int_{\{v=t\}\cap\{g_t=r\}\cap\Omega}|v-\tv|\,dl
        \leq\frac{K}{\sigma_0}\int_{\{v=t\}\cap\{g_t=r\}\cap\Omega}|J-\tJ|\,dl.
    \end{equation}
    By the coarea formula,
    \begin{align}\label{I4}
    \begin{split}
        \int_\R\int_\R\int_{\{v=t\}\cap\{g_t=r\}\cap\Omega}|J-\tJ|\,dl\,dr\,dt
        &=\int_\R\int_{\{v=t\}\cap\Omega}|\nabla g_t||J-\tJ|\,dS\,dt\\
        &\leq C_g\int_\R\int_{\{v=t\}\cap\Omega}|J-\tJ|\,dS\,dt\\
        &=C_g\int_\Omega|\nabla v||J-\tJ|\,dx\\
        &\leq\frac{C_g M}{\sigma_0}\int_\Omega|J-\tJ|\,dx\\
        &\leq\frac{C_g M}{\sigma_0}\cdot C\|H-\tH\|_{L^\infty(\Omega)}^{1/2}
    \end{split}
    \end{align}
    where we have used Theorem \ref{thm H1} and $C=C(\Omega, m,M,k_1,\sigma_1).$ By \eqref{I1},\eqref{I2},\eqref{I3}, and \eqref{I4}, the proof is complete.
    \hfill$\Box$

\section{$W^{1,1}$ stability of the minimizers}

In this section, we prove the stability of the minimizers of \eqref{mainmin} in $W^{1,1}.$ As mentioned in Section 2, in general, \eqref{mainmin} does not have unique minimizers, so to prove stability results, it is natural to expect stronger assumptions on the minimizers.

\begin{lemma}\label{lemma H3}
    Let $n=2,3,$ and suppose that $u,\tu\in H_0^1(\Omega)$ are admissible. Suppose $\partial\Omega$ is of class $C^3,$ $F\in C^2(\bar\Omega),$ $H,\tH\in C^1(\bar\Omega),$ and $\sigma,\tsigma\in C^{2,1}(\bar\Omega)$ with
    \begin{equation}\label{FHineq}
    \|F\|_{C^2(\Omega)},\|H\|_{C^1(\Omega)},\|\tH\|_{C^1(\Omega)}\leq k_2
\end{equation}
and
\begin{equation}\label{sigmaineq}
    \|\sigma\|_{C^2(\Omega)},\|\tsigma\|_{C^2(\Omega)}\leq\sigma_2
\end{equation}
for some $k_2,\,\sigma_2>0.$ Let
\begin{equation}\label{G}
    G(x)=\frac{\tJ(x)-J(x)}{\tsigma(x)},\,x\in\Omega,
\end{equation}
with $G=(G_1,G_2)$ for $n=2$ and $G=(G_1,G_2,G_3)$ for $n=3.$ Then
$$
    \|\nabla G_i\|_{L^1(\Omega)}\leq C\|J-\tJ||_{L^1(\Omega)}^{1/2},\hspace{0.1in}1\leq i\leq n
$$
for some constant $C\left(\Omega,k_2,\sigma_0,\sigma_2\right).$
\end{lemma}
{\bf Proof.}
Since $\Omega$ is a Lipschitz domain, then by the \textit{Gagliardo-Nirenberg interpolation inequality,}
    \begin{equation}\label{j1}
        \|\nabla G_i\|_{L^1(\Omega)}\leq C_1\left(\|D^2G_i\|_{L^1(\Omega)}^{1/2}\|G_i\|_{L^1(\Omega)}^{1/2}+\|G_i\|_{L^1(\Omega)}\right)
    \end{equation}
    for some $C_1$ depending on $\Omega.$    
    By \eqref{EL}, we have that $\nabla\cdot(\sigma(Du+F))=H.$ Equivalently, we have the elliptic PDE
    $$
        \nabla\cdot(\sigma Du)=H-\nabla\cdot(\sigma F).
    $$
    Note that $H-\nabla\cdot(\sigma F)\in H^1(\Omega).$ By elliptic regularity (Theorem 8.13 in \cite{GT}), $u,\tilde u\in H^3(\Omega),$ and
    $$
        \|u\|_{H^3(\Omega)}\leq C_2\left(\|u\|_{L^2(\Omega)}+\|H-\nabla\cdot(\sigma F)\|_{H^1(\Omega)}\right)
    $$
    where $C_2$ depends on $\sigma_0,\sigma_2,$ and $\partial\Omega.$ We have a similar bound for $\|\tu\|_{H^3\Omega}.$ Therefore, since $u\in C^1(\bar\Omega),$
    $$
        \|u\|_{H^3(\Omega)},\|\tu\|_{H^3(\Omega)}\leq C_3
    $$
    for some $C_3$ depending on $ k_2,\sigma_0,\sigma_2,$ and $\partial\Omega.$
    Denote $F=(F_1,F_2)$ if $n=2$ and $F=(F_1,F_2,F_3)$ if $n=3.$ Note that
    $$
        G_i=\tu_{x_i}+F_i-\frac{\sigma}{\tsigma}\left(u_{x_i}+F_i\right)
    $$
    for which it follows that
    \begin{equation}\label{j2}
        \|D^2G_i\|_{L^1(\Omega)}\leq|\Omega|^{1/2}\|D^2G_i\|_{L^2(\Omega)}\leq C_4
    \end{equation}
    for some constant $C_4$ depending on depending on $k_2, \sigma_0,\sigma_2,$ and $\Omega.$ Combining \eqref{j1} and $\eqref{j2}$ with
    $$
        \|G_i\|_{L^1(\Omega)}\leq\frac{\|J-\tJ\|_{L^1(\Omega)}}{\sigma_0},
    $$
    we obtain the desired result.
\hfill$\Box$

\bigskip

In order to prove that $u$ and $\tu$ are close in $W^{1,1}
(\Omega),$ we need
additional assumptions on the structure of level sets of $v.$

\begin{definition}\label{def H3}
Suppose $u$ is admissible and $v=u+f.$ Suppose $n=2,$ and $x\in\Omega.$ Choose $h\in S^1$ and $t\inR$ small enough so that $x+th\in\Omega.$ Let $\Gamma$ be the level set of $v$ containing $x$ and $\Gamma_t$ be the level set of $v$ containing $x+th.$ Let $\gamma,$ $\gamma_t$ be the parameterizations by arc length of $\Gamma$ and $\Gamma_t,$ respectively, where $\gamma(0),\gamma_t(0)\in\partial\Omega.$

Similarly, if $n=3,$ suppose that the level sets of $v$ can be foliated to one-dimensional curves as in Definition \ref{def H2}. Suppose $x\in\Omega$ and $h\in S^2.$ Choose $t$ small enough such that $x+th\in\Omega.$ Let $\Gamma$ and $\Gamma_t$ be the level sets of the form $\{v=\tau\}\cap\{g_\tau=r\}$ where $\tau,r\inR$ which contain $x$ and $x+th,$ respectively. Let $\gamma$ and $\gamma_t$ be the parameterizations by arc length of $\Gamma$ and $\Gamma_t,$ respectively where $\gamma(0),\gamma_t(0)\in\partial\Omega.$

We say that the level sets of $v$ are well structured if the following conditions are satisfied:
\begin{enumerate}[(a)]
    \item There exists $\tK\geq0$ such that
    \begin{equation}
        \left|\frac{\gamma_t'(s)-\gamma'(s)}{t}\right|\leq \tK
    \end{equation}
    for every $s\in[0,L],$ $t\inR,$ $x\in\Omega$ and $h\in S^{n-1}.$ In particular,
    \begin{equation}
        \gamma_t'(s)\to\gamma'(s)\text{ as }t\to0,
    \end{equation}
    where $\gamma'(s)=\frac{d\gamma}{ds}$ and $\gamma_t'(s)=\frac{d\gamma_t}{ds}.$
    \item There exists a bounded function $B_{x,h}(s)=B(x,h;s)\in L^{\infty}(\Omega\times S^{n-1}\times[0,K])$ such that
    \begin{equation}
        \lim\limits_{t\to0}\frac{\gamma_t(s)-\gamma(s)}{t}=B_{x,h}(s)
    \end{equation}
    for every $s\in[0,L],$ $x\in\Omega,$ and $h\in S^{n-1}.$
\end{enumerate}
\end{definition}

\begin{theorem}\label{thm H4}
    Let $n=2$ and suppose that $u$ and $\tu$ are admissible with $u|_{\partial\Omega}=\tu|_{\partial\Omega}=0.$ Suppose $F\in C^2(\bar\Omega),$ $H,\tH\in C^1(\bar\Omega),$ $\sigma,\tsigma\in C^2(\bar\Omega)$ and satisfy \eqref{FHineq} and \eqref{sigmaineq}. If the level sets of $v$ are well-structured in the sense of Definition \ref{def H3}, then
\begin{equation}
    \|\nabla u-\nabla\tu\|_{L^1(\Omega)}\leq C\|\,H-\tH\|_{L^\infty(\Omega)}^{1/4},
\end{equation}
for some constant $C=C(\Omega,m,M,k_1,k_2,K,\tK,\sigma_0,\sigma_1,\sigma_2)$ independent of $\tu$ and $\tsigma.$
\end{theorem}
{\bf Proof.}
    Fix $x\in\Omega$ and $h\in S^1$ and define
    $$
        \CL(x,h)
        :=(\nabla\tv(x)-\nabla v(x))\cdot h
        =\lim\limits_{t\to0}\Frac{[\tv(x+th)-v(x+th)]-[\tv(x)-v(x)]}{t}.
    $$

    Since the level sets of $v$ reach the boundary $\partial\Omega,$ there exist $y,y_t\in\partial\Omega$ such that
    $$
        v(x)=v(y)=\tv(y)=f(y)
    $$
    and
    $$
        v(x+th)=v(y_t)=\tv(y_t)=f(y_t).
    $$
    Thus,
    $$
        [\tv(x+th)-v(x+th)]-[\tv(x)-v(x)]
        =[\tv(x+th)-\tv(y_t)]-[\tv(x)-\tv(y)].
    $$
    Consider $\gamma$ and $\gamma_t,$ curves passing through $x$ and $x+th,$ as in Definition \ref{def H3} with $\gamma(0)=y$ and $\gamma_t(0)=y_t.$ Suppose $\gamma(s_0)=x$ and $\gamma_t(s_0)=x+th$ (where $\gamma_t$ is reparametrized if necessary). Then
    \begin{align*}
        [\tv(x+th)-v(x+th)]-[\tv(x)-v(x)]
        &=[\tv(x+th)-\tv(y_t)]-[\tv(x)-\tv(y)]\nl
        &=[\tv(\gamma_t(s_0))-\tv(\gamma_t(0))]-[\tv(\gamma(s_0))-\tv(\gamma(0))]\nl
        &=\int_0^{s_0}\nabla\tv(\gamma_t(s))\cdot\gamma_t'(s)\,ds-\int_0^{s_0}\nabla\tv(\gamma(s))\cdot\gamma'(s)\,ds.
    \end{align*}
    Thus,
    $$
        \CL(x,h)=\lim\limits_{t\to0}\,\recip{t}\left(\int_0^{s_0}\nabla\tv(\gamma_t(s))\cdot\gamma_t'(s)\,ds-\int_0^{s_0}\nabla\tv(\gamma(s))\cdot\gamma'(s)\,ds\right).
    $$
    On $\Gamma,$ we have $\frac{d}{ds}(v(\gamma(s)))
    =\nabla v(\gamma(s))\cdot\gamma'(s)
    =0
    =\frac{\sigma(\gamma(s))}{\tsigma(\gamma(s))}\nabla v(\gamma(s))\cdot\gamma'(s)
    =\frac{J(\gamma(s))}{\tsigma(\gamma(s))}\cdot\gamma'(s).$ Thus, $\frac{J(\gamma(s))}{\tsigma(\gamma(s))}\cdot\gamma'(s)=0.$ Similarly, on $\Gamma_t,$ we have
    $
        \frac{J(\gamma_t(s))}{\tsigma(\gamma_t(s))}\cdot\gamma_t'(s)=0.
    $
    Therefore,
    \begin{align*}
        \CL(x,h)
        &=\lim\limits_{t\to0}\,\recip{t}\left(\int_0^{s_0}\frac{\tJ(\gamma_t(s))}{\tsigma(\gamma_t(s))}\cdot\gamma_t'(s)\,ds-\int_0^{s_0}\frac{\tJ(\gamma(s))}{\tsigma(\gamma(s))}\cdot\gamma'(s)\,ds\right)\nl
        &=\lim\limits_{t\to0}\,\recip{t}\left(\int_0^{s_0}\frac{\tJ(\gamma_t(s))-J(\gamma_t(s))}{\tsigma(\gamma_t(s))}\cdot\gamma_t'(s)\,ds-\int_0^{s_0}\frac{\tJ(\gamma(s))-J(\gamma(s))}{\tsigma(\gamma(s))}\cdot\gamma'(s)\,ds\right).
    \end{align*}
    Define
    $$
        G(x):=\frac{\tJ(x)-J(x)}{\tsigma(x)},\,x\in\Omega.
    $$
    Then
    $$
        \CL(x,h)
        =\lim\limits_{t\to0}\,\recip{t}\left(\int_0^{s_0}G(\gamma_t(s))\cdot\gamma_t'(s)\,ds-\int_0^{s_0}G(\gamma(s))\cdot\gamma'(s)\,ds\right)=\lim\limits_{t\to0}\mathcal{G}(t)
    $$
    where
    \begin{equation}\label{CG}
        \mathcal{G}(t):=\recip{t}\int_0^{s_0}\left(G(\gamma_t(s))-G(\gamma(s))\right)\cdot\gamma_t'(s)\,ds
        +\recip{t}\int_0^{s_0}G(\gamma(s))\cdot(\gamma_t'(s)-\gamma'(s))\,ds.
    \end{equation}
    By Definition \ref{def H3}, there exists a positive constant $\tK$ such that $\left|\Frac{\gamma_t'(s)-\gamma'(s)}{t}\right|\leq \tK.$ Thus we obtain the following bound for the second integral in \eqref{CG}:
    \begin{equation}\label{estG2}
        \left|\recip{t}\int_0^{s_0}G(\gamma(s))\cdot(\gamma_t'(s)-\gamma'(s))\,ds\right|\leq\frac{\tK}{\sigma_0}\int_0^L|\tJ(\gamma(s)-J(\gamma(s))|\,ds.
    \end{equation}
    Next, by Definition \ref{def H3},
    $$
        \lim\limits_{t\to0}\frac{\gamma_t(s)-\gamma(s)}{t}=B_{x,h}(s).
    $$
    If $G=(G_1,G_2),$ then for $i\in\{1,2\},$
    $$
        \lim\limits_{t\to0}\frac{G_i(\gamma_t(s))-G_i(\gamma(s))}{t}
        =\lim\limits_{t\to0}\frac{G_i(\gamma(s)+tB_{x,h}(s))-G_i(\gamma(s))}{t}
        =\nabla G_i(\gamma(s))\cdot B_{x,h}(s).
    $$
    Thus, we estimate the first integral in \eqref{CG} by
    \begin{align}
        &\lim\limits_{t\to\infty}\recip{t}\int_0^{s_0}\big(G(\gamma_t(s))-G(\gamma(s))\big)\cdot\gamma_t'(s)\,ds\nonumber\nl
        &=\int_0^{s_0}\big(\nabla G_1(\gamma(s))\cdot B_{x,h}(s), \nabla G_2(\gamma(s))\cdot B_{x,h}(s)\big)\cdot\gamma_t'(s)\,ds\nonumber\nl
        &\leq\|B\|_{L^\infty(\Omega)}\int_0^{s_0}\big(|\nabla G_1(\gamma(s)|+|\nabla G_2(\gamma(s)|\big)\,|\gamma'(s)|\,ds\nonumber\nl
        &\leq\|B\|_{L^\infty(\Omega)}\int_0^L\big(|\nabla G_1(\gamma(s)|+|\nabla G_2(\gamma(s)|\big)\,ds\label{estG1}
    \end{align}
    By \eqref{estG2} and \eqref{estG1}, we conclude that 
    \begin{align*}
        |\nabla\tv(x)-\nabla v(x)|
        \leq\sup\limits_{h\in S^1}\CL(x,h)
        &\leq\frac{\tK}{\sigma_0}\int_0^L|\tJ(\gamma(s)-J(\gamma(s))|\,ds\nl
        &\hspace{0.6in}+\|B\|_{L^\infty(\Omega)}\int_0^L\big(|\nabla G_1(\gamma(s))|+|\nabla G_2(\gamma(s))|\big)\,ds.
    \end{align*}
    Thus,
    $$
        \int_\Gamma|\nabla\tv-\nabla v|\,dl
        \leq\frac{K\tK}{\sigma_0}\int_\Gamma|J-\tJ|\,dl+K\|B\|_{L^\infty(\Omega)}\int_\Gamma\big(|\nabla G_1|+|\nabla G_2|\big)\,dl. 
    $$
    Therefore, for each $\tau\inR,$
    \begin{align}\label{J1}
    \begin{split}
        \int_{\{v=\tau\}\cap\Omega}|\nabla\tv-\nabla v|\,dl
        &\leq\frac{K\tK}{\sigma_0}\int_{\{v=\tau\}\cap\Omega}|J-\tJ|\,dl\nl
        &\hspace{0.6in}+K\|B\|_{L^\infty(\Omega)}\int_{\{v=\tau\}\cap\Omega}\big(|\nabla G_1|+|\nabla G_2|\big)\,dl.
    \end{split}
    \end{align}
    By \eqref{J1} and the coarea formula,
    \begin{align*}
        \frac{m}{\sigma_1}\|\nabla v-\nabla\tv\|_{L^1(\Omega)}
        &\leq \int_{\Omega}|\nabla v|\,|\nabla v-\nabla\tv|\,dx\nl
        &=\int_\R\int_{\{v=\tau\}\cap\Omega} |\nabla v-\nabla\tv|\,dl\,d\tau\nl
        &\leq\frac{K\tK}{\sigma_0}\int_\R\int_{\{v=\tau\}\cap\Omega}|J-\tJ|\,dl\,d\tau\nl
        &\hspace{0.6in}+K\|B\|_{L^\infty(\Omega)}\int_\R\int_{\{v=\tau\}\cap\Omega}\big(|\nabla G_1|+|\nabla G_2|\big)\,dl\,d\tau\nl
        &\leq\frac{K\tK M}{\sigma_0^2}\int_\R\int_{\{v=\tau\}\cap\Omega}\frac{|J-\tJ|}{|\nabla v|}\,dl\,d\tau\nl
        &\hspace{0.8in}+\frac{KM\|B\|_{L^\infty(\Omega)}}{\sigma_0}\int_\R\int_{\{v=\tau\}\cap\Omega}\frac{|\nabla G_1|+|\nabla G_2|}{|\nabla v|}\,dl\,d\tau\nl
        &=\frac{K\tK M}{\sigma_0^2}\int_\Omega|J-\tJ|\,dx
        +\frac{KM\|B\|_{L^\infty(\Omega)}}{\sigma_0}\int_\Omega \big(|\nabla G_1|+|\nabla G_2|\big)\,dx\nl
        &\leq\frac{K\tK M}{\sigma_0^2}\|J-\tJ\|_{L^1(\Omega)}+\frac{KM\|B\|_{L^\infty(\Omega)}}{\sigma_0}\cdot2C_1\|J-\tJ\|_{L^1(\Omega)}^{1/2}\nl
        &\leq\left(\frac{K\tK M}{\sigma_0^2}(2M|\Omega|)^{1/2}+\frac{KM\|B\|_{L^\infty(\Omega)}}{\sigma_0}\cdot2C_1\right)\|J-\tJ\|_{L^1(\Omega)}^{1/2}
    \end{align*}
    where $C_1=C_1\left(\Omega,k_2, \sigma_0,\sigma_2\right)$ is the constant obtained from Lemma \ref{lemma H3}. By Theorem \ref{thm H1}, we obtain the desired result.
    \hfill$\Box$

    \begin{theorem}\label{thm H5}
    Let $n=3$ and suppose that $u$ and $\tu$ are admissible with $u|_{\partial\Omega}=\tu|_{\partial\Omega}=0.$ Suppose $F\in C^2(\bar\Omega),$ $H,\tH\in C^1(\bar\Omega),$ $\sigma,\tsigma\in C^2(\bar\Omega)$ and satisfy \eqref{FHineq} and \eqref{sigmaineq}. If the level sets of $v$ can be foliated to one-dimensional curves in the sense of Definition \ref{def H2} and the level sets of $v$ are well-structured in the sense of Definition \ref{def H3}, then
\begin{equation}
    \|\nabla u-\nabla\tu\|_{L^1(\Omega)}\leq C\|\,H-\tH\|_{L^\infty(\Omega)}^{1/4},
\end{equation}
for some constant $C=C(\Omega,m,M,k_1,k_2,K,\tK,\sigma_0,\sigma_1,\sigma_2)$ independent of $\tu$ and $\tsigma.$
\end{theorem}
{\bf Proof.}
    Similar to the proof of Theorem \ref{thm H4}, we can conclude that
    \begin{equation}\label{J2}
        \int_{V_{\tau,r}}|\nabla v-\nabla\tv|\,dl\leq \frac{K\tK}{\sigma_0}\int_{V_{\tau,r}}|J-\tJ|\,dl+K\|B\|_{L^\infty(\Omega)}\int_{V_{\tau,r}}\big(|G_1|+|G_2|+|G_3|\big)\,dl
    \end{equation}
    where $V_{\tau,r}:=\{v=\tau\}\cap\{g_\tau=r\}\cap\Omega$ (see Definiton \ref{def H2}) and $G=(G_1,G_2,G_3)$ is defined in \eqref{G}. It follows from \eqref{J2} and the coarea formula that
    \allowdisplaybreaks
    \begin{align*}
        \frac{m}{\sigma_1}\|\nabla v-\nabla\tv\|_{L^1(\Omega)}
        &\leq\int_\Omega|\nabla v|\,|\nabla v-\nabla\tv|\,dx\nl
        &=\int_\R\int_{\{v=\tau\}\cap\Omega}|\nabla v-\nabla\tv|\,dS\,d\tau\nl
        &\leq\recip{c_g}\int_\R\int_{\{v=\tau\}\cap\Omega}|\nabla g_\tau|\,|\nabla v-\nabla\tv|\,dS\,d\tau\nl
        &=\recip{c_g}\int_\R\int_\R\int_{V_{\tau,r}}|\nabla v-\nabla\tv|\,dl\,dr\,d\tau\nl
        &\leq\frac{K\tK}{\sigma_0c_g}\int_\R\int_\R\int_{V_{\tau,r}}|J-\tJ|\,dl\,dr\,d\tau\nl
        &\hspace{0.44in}+\frac{K\|B\|_{L^\infty(\Omega)}}{c_g}\int_\R\int_\R\int_{V_{\tau,r}}\big(|\nabla G_1|+|\nabla G_2|+|\nabla G_3|\big)\,dl\,dr\,d\tau\nl
        &\leq\frac{K\tK}{\sigma_0c_g}\cdot C_g\cdot\frac{M}{\sigma_0}\int_\R\int_\R\int_{V_{\tau,r}}\frac{|J-\tJ|}{|\nabla g_\tau|\,|\nabla v|}\,dl\,dr\,d\tau\nl
        &\hspace{0.44in}+\frac{K\|B\|_{L^\infty(\Omega)}}{c_g}\cdot C_g\cdot\frac{M}{\sigma_0}\int_\R\int_\R\int_{V_{\tau,r}}\frac{|\nabla G_1|+|\nabla G_2|+|\nabla G_3|}{|\nabla g_\tau|\,|\nabla v|}\,dl\,dr\,d\tau\nl
        &=\frac{K\tK MC_g}{\sigma_0^2c_g}\int_\R\int_{\{v=\tau\}\cap\Omega}\frac{|J-\tJ|}{|\nabla v|}\,dS\,d\tau\nl
        &\hspace{0.44in}+\frac{KMC_g\|B\|_{L^\infty(\Omega)}}{\sigma_0c_g}\int_\R\int_{\{v=\tau\}\cap\Omega}\frac{|\nabla G_1|+|\nabla G_2|+|\nabla G_3|}{|\nabla v|}\,dS\,d\tau\nl
        &=\frac{K\tK MC_g}{\sigma_0^2c_g}\int_\Omega|J-\tJ|\,dx\nl
        &\hspace{0.44in}+\frac{KMC_g\|B\|_{L^\infty(\Omega)}}{\sigma_0c_g}\int_\Omega\big(|\nabla G_1|+|\nabla G_2|+|\nabla G_3|\big)\,dx\nl
        &\leq\frac{K\tK MC_g}{\sigma_0^2c_g}\|J-\tJ\|_{L^1(\Omega)}+\frac{KMC_g\|B\|_{L^\infty(\Omega)}}{\sigma_0c_g}\cdot3C_1\|J-\tJ\|_{L^1(\Omega)}^{1/2}\nl
        &\leq\left(\frac{K\tK MC_g}{\sigma_0^2c_g}(2M|\Omega|)^{1/2}+\frac{KMC_g\|B\|_{L^\infty(\Omega)}}{\sigma_0c_g}\cdot 3C_1\right)\|J-\tJ\|_{L^1(\Omega)}^{1/2}
    \end{align*}
    where $C_1=C_1\left(\Omega,k_2,\sigma_0,\sigma_2\right)$ is the constant obtained from Lemma \ref{lemma H3}. By Theorem \ref{thm H1}, we obtain the desired result.
    \hfill$\Box$

\begin{theorem}\label{thm H6}
Let $n=2$ and suppose that $u$ and $\tu$ are admissible with $u|_{\partial\Omega}=\tu|_{\partial\Omega}=0.$ Suppose $F\in C^2(\bar\Omega),$ $H,\tH\in C^1(\bar\Omega),$ $\sigma,\tsigma\in C^2(\bar\Omega)$ and satisfy \eqref{FHineq} and \eqref{sigmaineq}. If the level sets of $v$ are well-structured in the sense of Definition \ref{def H3} then
\begin{equation}
    \|\sigma-\tsigma\|_{L^1(\Omega)}\leq C\|\,H-\tH\|_{L^\infty(\Omega)}^{1/4},
\end{equation}
for some constant $C=C(\Omega,m,M,k_1,k_2,K,\tK,\sigma_0,\sigma_1,\sigma_2)$ independent of $\tsigma.$
\end{theorem}
{\bf Proof.}
    \begin{align*}
        \Int_\Omega|\sigma-\tsigma|\,dx
        =\Int_\Omega\left|\Frac{a}{|\nabla v|}-\Frac{a}{|\nabla \tv|}\right|\,dx
        =\Int_\Omega a\Frac{||\nabla v|-|\nabla\tv||}{|\nabla v||\nabla\tv|}\,dx\leq\Frac{M\sigma_1^2}{m^2}\Int_\Omega|\nabla v-\nabla\tv|\,dx
    \end{align*}
    and we apply Theorem \ref{thm H4}.
\hfill$\Box$

\bigskip

Lastly, the following theorem follows from Theorem \ref{thm H5} and a calculation similar to the proof of Theorem \ref{thm H6}.

\begin{theorem}
Let $n=3$ and suppose that $u$ and $\tu$ are admissible with $u|_{\partial\Omega}=\tu|_{\partial\Omega}=0.$ Suppose $F\in C^2(\bar\Omega),$ $H,\tH\in C^1(\bar\Omega),$ $\sigma,\tsigma\in C^2(\bar\Omega)$ and satisfy \eqref{FHineq} and \eqref{sigmaineq}. If the level sets of $v$ can be foliated to one-dimensional curves in the sense of Definition \ref{def H2}, and the level sets of $v$ are well-structured in the sense of Definition \ref{def H3}, then
\begin{equation}
    \|\sigma-\tsigma\|_{L^1(\Omega)}\leq C\|\,H-\tH\|_{L^\infty(\Omega)}^{1/4},
\end{equation}
for some constant $C=C(\Omega,m,M,k_1,k_2,K,\tK,\sigma_0,\sigma_1,\sigma_2)$ independent of $\tsigma.$
\end{theorem}

\subsection{Numerical Simulation}

It was proven in \cite{CMO} that the following alternating split Bregman algorithm converges weakly to a minimizer, $u,$ of \eqref{mainmin}.

\vspace{0.25in}

\textbf{Algorithm 1}

\vspace{0.3in}

Let $\lambda>0.$ Let $H\in L^\infty(\Omega)$ and initialize $b^0,d^0\in\left(L^2(\Omega)\right)^n.$ For $k\geq0,$
\begin{enumerate}
    \item Solve
    \begin{equation*}\label{Poisson}
        \Delta u^{k+1}=-\nabla\cdot\left(b^k-d^k\right)+\recip{\lambda}H,\enspace u^{k+1}\big|_{\partial\Omega}=0.
    \end{equation*}
    \item Compute
    $$
        d^{k+1}:=
        \begin{cases}
            \max\left\{\left|b^k+\nabla u^{k+1}+F\right|-\frac{a}{\lambda},0\right\}\frac{b^k+\nabla u^{k+1}+F}{\left|b^k+\nabla u^{k+1}+F\right|}-F &\text{if }\left|b^k+\nabla u^{k+1}+F\right|\neq0\\
            -F &\text{if }\left|b^k+\nabla u^{k+1}+F\right|=0
        \end{cases}.
    $$
    \item Let
    $$
        b^{k+1}:=b^k+\nabla u^{k+1}-d^{k+1}.
    $$
\end{enumerate}

Weitao Chen \cite{CMO} wrote MATLAB code to run numerical simulations demonstrating the convergence of Algorithm 1. We use a slightly modified version of the code to run the following simulation. We construct an example with a known minimizer and compare the approximate solution with the known exact solution. The following data gives the numerical errors of numerical simulations with Algorithm 1 (with $\lambda =1$) for mesh size $h=1/100.$ The iterations are stopped when 
\[|u^{k+1}-u^k|/|u^{k+1}|<1\times10^{-7}.\]

We examine the effect of noise in our simulation. The noise model we used is a simple stochastic model $\tH = H + \gamma * R,$ where $R$ is normally distributed pseudo-random matrix of the order as $H$ with mean zero and standard deviation of one, and $\gamma > 0$ is the model standard deviation chosen as $\gamma = \delta * ||H||/||R||,$ where $\delta$
is the noise level. Replacing $H$ with $\tH$ in Algorithm 1, we obtain the following approximations of $\tu$ and compare them with $u.$

\begin{example}
    Let $\Omega=(0,1)\times(0,1)\subset\R^2$ and $u:=xy(1-x)(1-y)$ so that $u|_{\partial\Omega}=0.$ Let $F:=-\nabla u+(1,x+y)$ so that $\nabla u+F=(1,x+y)$ and let $a:=|\nabla u+F|=\sqrt{1+(x+y)^2}.$ Let $H:=\nabla\cdot(1,x+y)=1$ so that \eqref{EL} holds. 
\end{example}

\begin{table}[ht]
\centering
    
    \label{tab:exampleH1}
\begin{tabular}{ | m{3.8cm}| m{5cm} | m{3.85cm} |} 
  \hline
  \vspace{0.1cm}
  Low Noise $(\delta =0.01)$ & Moderate Noise $(\delta=0.035)$ & High Noise $(\delta =0.06)$ \\ 
  \hline
  \vspace{0.1cm}
  $7.5368\times 10^{-4}$ & $0.0027$ & $0.0050$\\ 
  \hline
  \end{tabular}
  \caption{Relative $L^2$ errors for Algorithm 1 with $h=1/100$ and $Tol=1\times10^{-7}$ and increasing noise levels.}\vspace{0.1in}
  \end{table}

\begin{center}
    \includegraphics[width=4.2in]{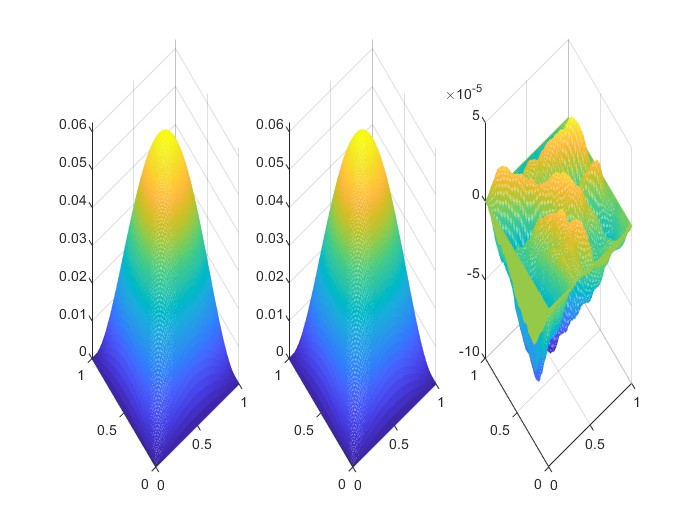}
        \begin{minipage}{15cm}

        \small  Figure 1: Numerical approximation $\tilde{u}^{313}$ (Left), Exact $u$ (Middle), and Error\\ $\tilde{u}^{313}-u$ (Right) with low noise. Maximum error: $8.0267\times10^{-5}.$

    \end{minipage}
\end{center}

\begin{center}
    \includegraphics[width=4.2in]{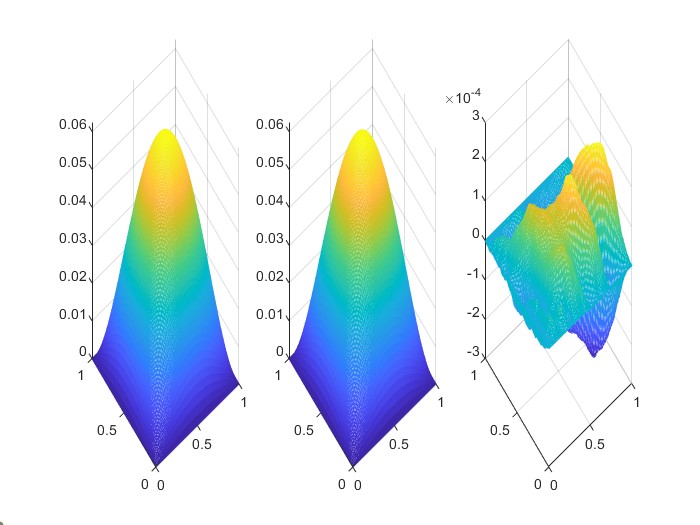}
        \begin{minipage}{15cm}

        \small  Figure 2: Numerical approximation $\tilde{u}^{324}$ (Left), Exact $u$ (Middle), and Error\\ $\tilde{u}^{324}-u$ (Right) with moderate noise. Maximum error: $2.8077\times10^{-4}.$

    \end{minipage}
\end{center}

\begin{center}
    \includegraphics[width=4.2in]{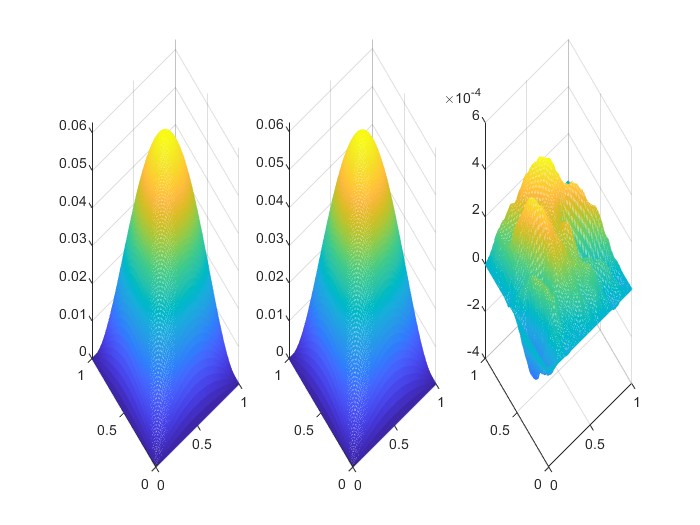}
        \begin{minipage}{15cm}

        \small  Figure 3: Numerical approximation $\tilde{u}^{323}$ (Left), Exact $u$ (Middle), and Error\\ $\tilde{u}^{323}-u$ (Right) with high noise. Maximum error: $4.6268\times10^{-4}.$

    \end{minipage}
\end{center}

\begin{center}
    \includegraphics[width=4.2in]{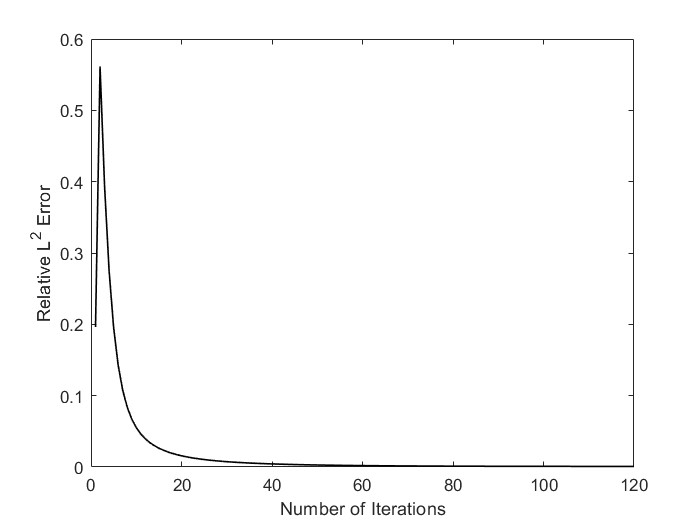}
        \begin{minipage}{15cm}

        \small  Figure 4: Rate of convergence for Algorithm 1 with low noise.

    \end{minipage}
\end{center}

\begin{center}
    \includegraphics[width=4.2in]{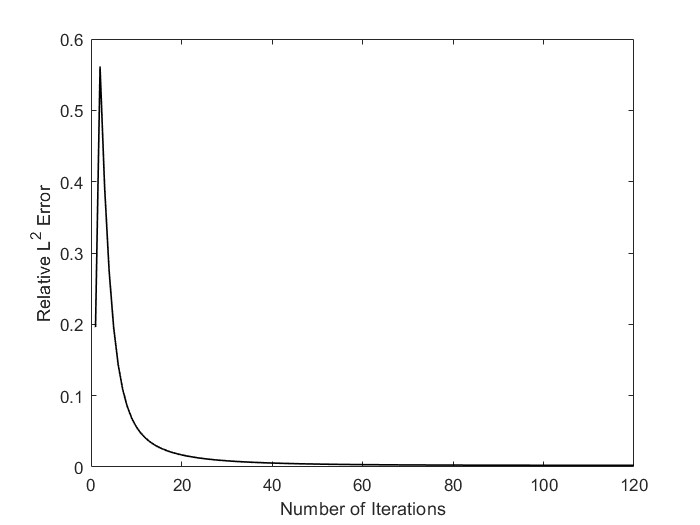}
        \begin{minipage}{15cm}

        \small  Figure 5: Rate of convergence for Algorithm 1 with moderate noise.

    \end{minipage}
\end{center}

\begin{center}
    \includegraphics[width=4.2in]{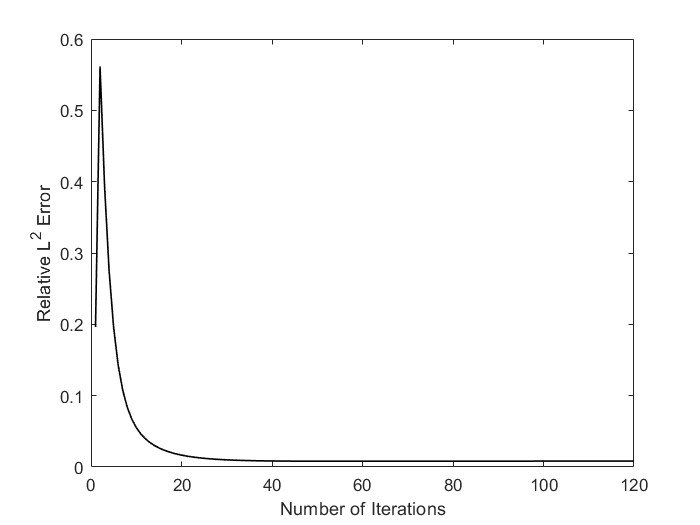}
        \begin{minipage}{15cm}

        \small  Figure 6: Rate of convergence for Algorithm 1 with high noise.

    \end{minipage}
\end{center}

\newpage

\textbf{Data availability.} This study is theoretical and includes numerical simulations. No external datasets were used. The data generated during the simulations are available from the authors upon reasonable request.

\end{document}